\documentclass[a4paper,12pt]{article}

\usepackage{mathrsfs}
\usepackage{ifthen}
\usepackage{verbatim}
\usepackage[english]{babel}
\usepackage{fancyhdr}
\usepackage{enumitem}
\usepackage{amsmath,amssymb,amsthm,graphicx}
\usepackage{latexsym}
\usepackage{ifpdf}
\usepackage{appendix}
\usepackage{color}
\usepackage{bbm}

\usepackage{tikz}
\usetikzlibrary{matrix,arrows,decorations.pathmorphing}

\newtheorem{Theorem}{Theorem}[section]
\newtheorem{Lemma}[Theorem]{Lemma}

\newtheorem{Proposition}[Theorem]{Proposition}

\newtheorem{Definition}[Theorem]{Definition}

\newtheorem*{Notation}{Notation}

\theoremstyle{definition}
\newtheorem{Remark}[Theorem]{Remark}
\newtheorem{Example}[Theorem]{Example}

\def\l{\left}
\def\r{\right}
\newcommand{\seq}[1]{\l\{#1\r\}}
\newcommand{\bra}[1]{\l(#1\r)}

\newcommand{\abs}[1]{\l|#1\r|}
\newcommand{\qv}[1]{\langle #1 \rangle}
\newcommand{\ip}[2]{\langle #1, #2 \rangle}

\def\norm#1{\left \Vert #1 \right \Vert}

\renewcommand{\hat}{\widehat}
\renewcommand{\bar}{\overline}
\def\gap{\; \; \;}

\newcommand{\EE}[1]{\mathbb{E}\ex{#1}}

\newcommand{\RR}{\mathbb{R}}

\def\id{\textbf{\textup{1}}}

\newcommand{\rp}[1]{\mathbb{{#1}}}
\newcommand{\Dd}{\mathcal{D}}

\newcommand{\pnorm}[2]{\norm{#1}_{#2\textup{-var;}[0,1]}}

\newcommand{\dhol}[1]{d_{#1\textup{-H\"{o}l;}[0,1]}}
\newcommand{\dvar}[1]{d_{#1\textup{-var;}[0,1]}}

\newcommand{\floor}[1]{\lfloor{#1}\rfloor}

\newcommand{\Mm}{\mathcal{M}}

\newcommand{\RP}[1]{\textbf{\textup{#1}}}

\def\EE{\mathbb{E}}

\def\Mm{\mathcal{M}}

\def\XX{\mathbb{X}}

\usepackage[us]{datetime}
\usepackage{geometry}
\numberwithin{equation}{section}
\numberwithin{figure}{section}

\begin{document}
\title{Discretely sampled signals and the rough Hoff process}
\author{Guy Flint, Ben Hambly and Terry Lyons\\ \\ \small{\textit{ Mathematical Institute, University of Oxford, Woodstock Road, OX2 6GG, UK }}}
\date{\today}
\maketitle

\begin{changemargin}{0.3cm}{0.3cm} 
\abstract{
We introduce a canonical method for transforming a discrete sequential data set into an associated rough path made up of lead-lag increments. In particular, by sampling a $d$-dimensional continuous semimartingale $X:[0,1] \rightarrow \mathbb{R}^d$ at a set of times $D=\{t_i\}$, we construct a piecewise linear, axis-directed process $X^D: [0,1] \rightarrow
\mathbb{R}^{2d}$ comprised of a past and future component. 
We call such an object the Hoff process associated with the discrete data $\{X_{t}\}_{t_i\in D}$. The Hoff process can be lifted to its natural rough path enhancement and we consider the question of convergence as the sampling frequency increases. 
We prove that the It\^{o} integral can be recovered from a sequence of random ODEs driven by the components of $X^D$. This is in contrast to the usual Stratonovich integral limit suggested by the classical Wong-Zakai Theorem \cite{eugene1965relation}. 
Such random ODEs have a natural interpretation in the context of mathematical finance.}
\end{changemargin}


\smallskip
\noindent \texttt{Guy.Flint@maths.ox.ac.uk}

\section{Discrete streams of event data}

A stream of event data is a time-ordered sequence of events where we have the time of the event as a real value and the data associated with the event typically taking some numerical value. 
In many applications, for example order books in finance or internet traffic, it is typical that these events come in a variety of quite different types at different frequencies, and that the realization of such streams can have a complex dependency structure between their coordinates. However, if we focus our interest in such event streams around their effects and the actions they trigger, then rough path theory provides an established and effective method (using the signature of the stream and the It\^{o} map) to transform these event streams into much smaller data sets, while retaining adequate information to accurately approximate the primary effects of the data (see \cite[\S1-5]{lyons2014rough}). The focus of rough path theory is to give quantitative mathematical meaning to the concept of the evolving response $y=(y_t)$ of a system $f$ driven by a signal $x=(x_t)$, which we usually write in the form of a rough differential equation:
\begin{equation}\label{e-rde}
dy_t = f\bra{y_t} dx_t, \gap y_0 = a.
\end{equation}
Here, $x$ can have a very complex local structure (that is, it can be a rough path). The theory shows how to summarise the data $x$ while retaining a good approximation of $y$. A reader familiar with differential equations might want to think about the case where $x$ is smooth but highly oscillatory on normal scales. The meaning of the differential equation (\ref{e-rde}) is not in doubt. Indeed, rough path theory tackles the question of how to define solutions by exploiting the smoothness of the system $f$. The theory summarises $x$ on normal scales so as to effectively capture its effects and predict $y$ without knowing or caring further about $x$ on microscopic scales. 

This paper introduces a canonical method of transforming discrete numerical event data $\{X_{t_i}\}_{t_i\in D} \subset \RR^d$ into a geometric $2$-rough path $\XX^D$ which we call a Hoff process. Importantly, our transformation takes into account the order of events and the effect of latency between different coordinates. Moreover, differential and integral equations driven by $\XX^D$ have a classical as well as a rough path meaning, and we can consider the question of convergence as our sampling rate becomes finer by exploiting the It\^{o} map from rough path theory. 

With this in mind, this paper examines the interplay between this way of looking at discrete data through the Hoff process in the special case of when the discrete event data comes from a multidimensional semimartingale signal. In effect, we prove the equivalent of the Wong-Zakai Theorem by demonstrating that, under certain regularity conditions, a sequence of Hoff processes is Cauchy in the standard $p$-variation rough path metric as the sampling frequency becomes finer. This is the content of the first main result of the paper (Theorem \ref{t-hoff-process-convergence}). Before introducing the second result, let us formally define a Hoff process in the next section. 

\section{Defining the Hoff process}\label{s-hoff-define}

This section defines the Hoff process as a method of transforming discrete numerical event data in $\RR^d$ into a geometric $2$-rough path.
Given a sample $\{X_{t_i}\}_{t_i\in D} \subset \RR^d$ at times $D=\{t_i\}_{i=0}^n \subset [0,1]$, the standard approach to transforming this into a continuous time trajectory is to define the c\`{a}dl\`{a}g, piecewise constant path $Z:[0,1]\to\RR^d$ by $Z_t = X_{t_i}$ for $t\in [t_{i},t_{i+1})$. Although this is a natural construction, the order of events across different coordinates is in some sense lost. Following the PhD thesis \cite{hoff2005brownian} of B. Hoff, we go one step further and consider the continuous process made up of lead and lag versions of this standard interpolation. We call the resultant path $X^D:[0,1]\to\RR^{2d}$ the Hoff process:

\begin{Definition}\label{d-hoff}
Let $D=\{t_i\}_{i=0}^n \subset [0,1]$ be a sequence of times with $(t_0,t_n)=(0,1)$. 
Define the path $\hat{X}_t : [0,2n]\to\RR^{2d}$ given by 
\begin{align*}
\hat{X}^D_{t}&=\bra{\hat{X}^{D,b}_t,\hat{X}^{D,f}_t}\\
:&=\begin{cases}
\left(X_{t_{k}},X_{t_{k+1}}\right) & \gap t\in[2k,2k+1);\\
\left(X_{t_{k}},X_{t_{k+1}}+2(t-(2k+1))\left(X_{t_{k+2}}-X_{t_{k+1}}\right)\right) & \gap t\in[2k+1,2k+\frac{3}{2});\\
\left(X_{t_{k}}+2(t-(2k+\frac{3}{2}))\left(X_{t_{k+1}}-X_{t_{k}}\right),X_{t_{k+2}}\right) & \gap t\in[2k+\frac{3}{2},2k+2);
\end{cases}
\end{align*}
with $\hat{X}^D_n=(X_{t_{n-1}},X_{t_n})$.
We define the Hoff process associated with the sequential data $\seq{X_{t_i}}_{t_i\in D}$ to be the path $X^D : [0,1]\to\RR^{2d}$ given by
\[
X^D_t  = (X^{D,b}_t,X^{D,f}_t):= \hat{X}^D_{2nt}. 
\]
The rough Hoff process $\XX^D$ is the $2$-rough path lift of $X^D$ (to be defined in Section \ref{s-rough-path}). 
\end{Definition}

We call $X^{D,b}$ and $X^{D,f}$ the \textit{lag} and \textit{lead} components of the Hoff process respectively. Since lead and lag share the same first letter, we use $b$ for \textit{backward} and $f$ for \textit{forward}. At this point we find it helpful in building intuition about $X^D$ to include a diagram of a possible trajectory (see Figure 2.1). 

\begin{figure}\label{f-hoff-greg}
  \centering
    \includegraphics[width=14cm]{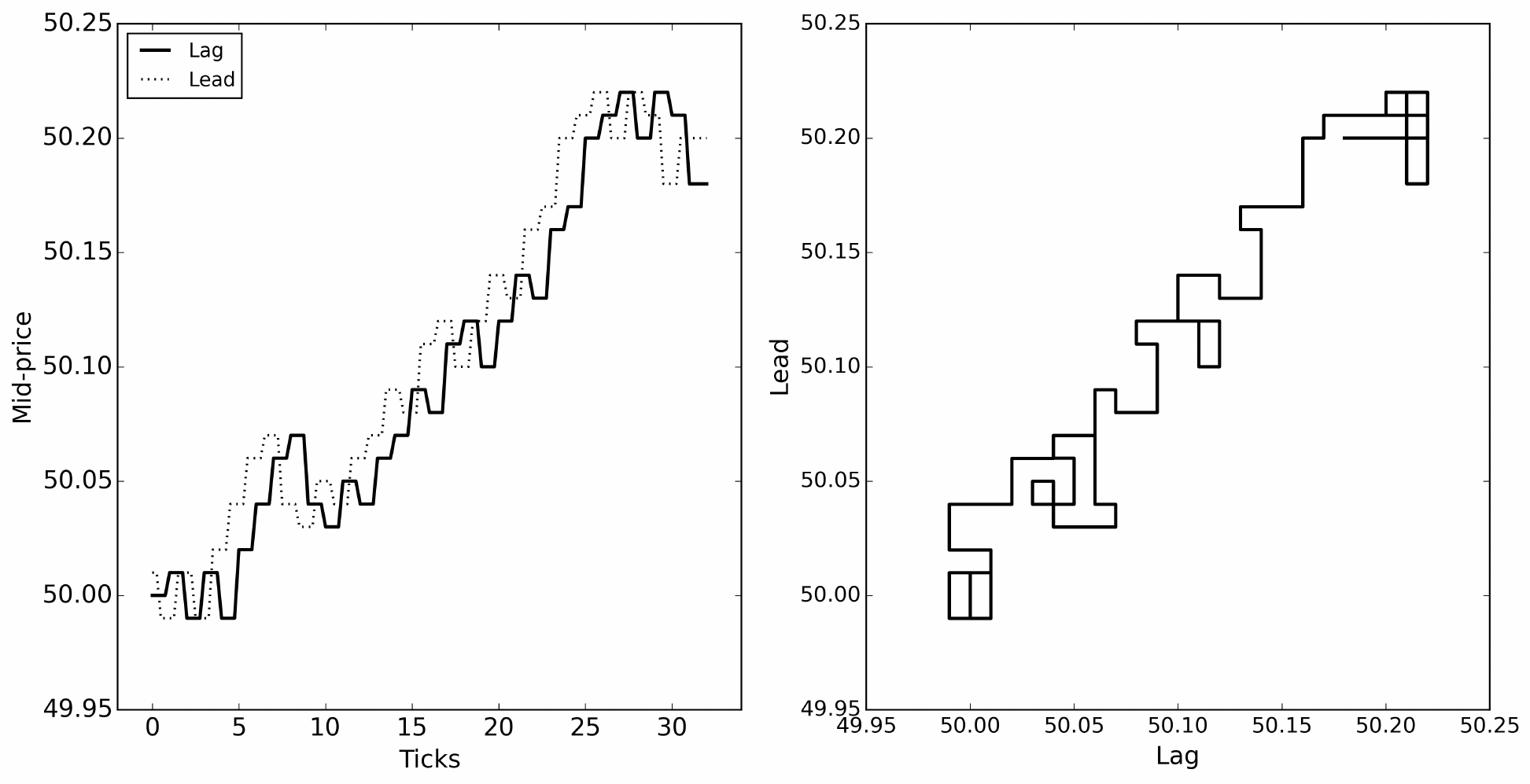}
      \caption{An example of a Hoff process trajectory}
\end{figure}

The definition states that $X_t=\bra{X_{t_k}, X_{t_{k+1}}}$ for $t\in [\frac{k}{n},\frac{2k+1}{2n})$, and on the intervals $[\frac{2k+1}{2n},\frac{k+1}{n})$, $X$ is a special interpolation between $(X_{t_{k}},X_{t_{k+1}})$ and $(X_{t_{k+1}}, X_{t_{k+2}})$. 
This interpolation is done by first updating the lead component, followed by the lag component.
We note that other lead-lag interpolations are possible (\cite{gergely2013extracting,learning,ni2015}) but we use Definition \ref{d-hoff} because of the following interpretation in the context of mathematical finance. 

Suppose one is  given the tick data of a collection of stocks $\{S_{t_i}\}_{i=0}^n \subset \RR^d$ at certain times $D=\{t_i\}_{i=0}^n$. For simplicity we assume that $(t_0,t_n)=(0,1)$.
Using the sequential data $\{X_{t_i}\}_{i=0}^n:=\{(t_i,S_{t_i})\}_{i=0}^n \subset [0,\infty) \times \RR^d \simeq \RR^{d+1}$, we could define the corresponding c\`{a}dl\`{a}g function $Z :t\in[0,1]\mapsto (Z^1_t,\ldots,Z^d_t)\in\RR^d$ as above. One could think of $Z$ as representing the current known price for the given stock collection. Similarly, we can think of the previsible process
\begin{equation}\label{e-previsible-z}
\tilde{Z}_t := \sum_{t_i \in D} Z_{t_i} \id{\seq{t\in (t_{i+1}, t_{i+2}]}},
\end{equation}
as recording the last known traded price for the collection. 
The performance of an agent's trading strategy could then be represented as the It\^{o} integral:
\begin{align}
Y^D:=
\int^1_0 f\bra{\tilde{Z}_{t}}\, dZ_t :&= \sum_{k=1}^d \int^1_0 f_k\bra{\tilde{Z}_{t}} \, dZ_t^k\notag\\
&= \sum^d_{k=1} \sum_{t_i \in D} f_k\bra{t_i, S_{t_i}}\bra{S^k_{t_{i+2}}-S^k_{t_{i+1}}},\label{e-agent}
\end{align}
where
for $f=(f_1,\ldots,f_d):\RR\times\RR^d\to\RR^d$. Here, $f_k(t_i,S_{t_i})$ represents the amount stock $k$ held by the agent at time $t_{i+1}$, given that they hold the portfolio $S_{t_i}$ at time $t_i$. Following the work of Friz in \cite{friz_examples}, we use (\ref{e-previsible-z}) instead of the well-known process $Z_{t-}=\sum_{t_i\in D} Z_{t_i}\id\{t \in (t_i,t_{i+1}]\}$ because we want the agent to act on the market at time $t_{i+1}$ based on their information at time $t_{i}$; there is a lag between receiving market information and acting upon it. Alternatively, one could interpret the model as representing a delay between placing an order on an exchange and the consequent change in the portfolio's value. Both events happen over disjoint time intervals. 

On the other hand, let $X^D\to\RR^{2(d+1)}$ denote the Hoff process associated with the same discrete data $\{X_{t_i}\}_{i=0}^n=\{(t_i,S_{t_i})\}_{i=0}^n\subset \RR^{d+1}$, with $X^{D,0}=(X^{D,b,0},X^{D,f,0})$ representing the time component. A simple computation gives
\begin{align*}
\sum_{t_i \in D} f_k\bra{t_i, S_{t_i}}\bra{S^k_{t_{i+2}}-S^k_{t_{i+1}}}
&=\sum_{i=0}^{n-2} f_k\bra{t_i, S_{t_i}}\bra{S^k_{t_{i+2}}-S^k_{t_{i+1}}}\\
&= \sum_{i=0}^{n-2} \int^{2k+\frac{3}{2}}_{2k+1} f_k\bra{\hat{X}^{D,b}_t}\, d\hat{X}^{D,f,k}_t\\
&= \sum^{n-2}_{i=0} \int^{2k+2}_{2k} f_k\bra{\hat{X}^{D,b}_t}\, d\hat{X}^{D,f,k}_t
= \int^1_0 f_k\bra{X^{D,b}_t}\, dX_t^{D,f,k}.
\end{align*}
That is, the performance of the agent's portfolio (\ref{e-agent}) can be rewritten as an \textit{exact integral} of the lag and lead components of the Hoff process. Importantly, since $X^D$ is continuous and certainly of finite variation, the last integral is well-defined as a classical Riemann-Stieltjes integral. Thus by using the particular interpolation scheme above, via the Hoff process, we are able to recast the simple forward Riemann sums of (\ref{e-agent}) as an exact integral of a piecewise linear path. 
We note in passing that an alternative pathwise integral in the context of mathematical finance via rough path theory can be found in the preprint \cite{perkowski2013pathwise} (the authors thank the referee for pointing out this paper). 

For the rest of this paper we will assume that the underlying event data arises from sampling a continuous semimartingale $X: t\in [0,1] \mapsto(X^1_t,\ldots,X^d_t) \in \RR^d$.
Then a natural question to ask is what does the integral $Y^D$ converge to as the sampling frequency increases; that is, as the mesh of our time partition $D$, $\abs{D}:=\max_{t_i\in D} \abs{t_{i+1}-t_i}$, goes to $0$. Such a question could arise in the situation where the agent has access to more data points, thus reducing the delay between the market behaviour and his strategy, (equivalently the delay between the forward and backward components of the Hoff process $X^D$). 

To answer this question, we use the theory of rough paths. Specifically, we lift the Hoff process $X^D$ to its natural rough path enhancement $\rp{X}^D$. 
The rough path lift $\rp{X}^D$, also known as the signature, is a well-defined transform in the sense than no information is lost (see \cite{hambly2005uniqueness}). The signed area enclosed by the different coordinates of $X^D$ is contained in $\rp{X}^D$. In particular, the area between the forward and backward components of $X^D$ is included and thus we have a natural candidate approximation for the quadratic variation of the original semimartingale $X$. Our first main result concerns the convergence of this area to this quadratic variation and the theorem more generally establishes the convergence of the rough Hoff process $\rp{X}^D$ to a rough path limit in the $p$-variation rough path topology as $\abs{D}\to 0$. As a corollary of this theorem, our second result proves that the limit of the random integrals $Y^D$ in (\ref{e-agent}) is the It\^{o} integral not the Stratonovich integral as might be expected from the classical Wong-Zakai Theorem (see the original paper \cite{eugene1965relation} by Wong and Zakai, or the rough path versions in \cite{coutin2005semi,sip1993}). This is almost an immediate consequence of the continuity of the It\^{o} map of rough path theory. The basic idea of the proof is that the areas between the forward and backward components of the Hoff process, (captured by $\rp{X}^D$), introduce a correction factor in the stochastic integral limit, consequently allowing us to recover the It\^{o}, not Stratonovich, integral. To be precise, Theorem \ref{t-ito-sde} implies the convergence
\begin{align*}
\sum_{k=1}^d \int^1_0 f_k\bra{X_t^{D,b}}\, dX_t^{D,f,k} \to 
\int^1_0 f\bra{X_t}\, dX_t
&:= \sum_{k=1}^d \int^1_0 f_k\bra{X_t}\, dX^k_t\\ 
&\textup{ as } \abs{D}:= \max_{t_i,t_{i+1}\in D} \abs{t_{i+1}-t_i} \to 0, 
\end{align*}
either in probability or some $L^p$-norm, (depending on the regularity imposed on $X$). 
We refer to \cite{friz2013physical,friz2009rough} and Theorem 17.20 of \cite{FV} for similar results concerning lead-lag driven random ODE convergence using rough path theory. 

\begin{Remark}
In a heuristic sense, this result demonstrates that the It\^{o} integral makes good sense as a classical integral of the Hoff process and does not need a probabilistic underpinning. 
As a consequence, we feel that despite the focus of the present paper in studying semimartingales as the underlying signal, our perspective and results undermine the concept to a significant extent. Taking mathematical finance as an example, almost all data presents itself initially in discrete form. The standard practice is to assume that the data arises from sampling a semimartingale signal so that It\^{o}'s theory of stochastic integration can be used to model hedging and investment strategies. It is identified in many contexts that this is a serious simplification. For example, momentum trading and long-range dependence may be better captured in a non-semimartingale context by modelling volatility using fractional Brownian motion with a Hurst parameter $H\neq\frac{1}{2}$ (\cite{bayer2015pricing,comte1998long,gatheral2014volatility}). We can now see that by moving the data to the Hoff process, now regarded as a geometric rough path, we can capture our data in a model that further allows the simulation of various strategies, thus enabling simplification and summary, without making any unrealistic assumptions about the nature of the data as a semimartingale. The pathwise concept of volatility is replaced by \textit{area} and is naturally included in any second order description of the Hoff process as a rough path. Successful modelling of financial data needs to understand the statistics of the first few terms in the signature of the Hoff process. In general these will be incompatible with any semimartingale model.
\end{Remark}

The recent papers \cite{gergely2013extracting,learning,lyons2014feature,ni2015} have also taken the approach of viewing time series data as samples of an underlying continuous process, which can in turn be lifted to a rough path. In \cite{gergely2013extracting} the authors have a very similar definition of a lead-lag path and apply standard regression analysis to linear functionals of the resultant rough path in order to extract interesting features of WTI crude oil futures, (such as deciding whether a given sample trade was made in the morning or afternoon).

\subsection*{Outline and notation of the paper}

We have attempted to make the paper as self-contained as possible. Before presenting the main results in Sections 3 and 4, we first use the next section to introduce the relevant elements of rough path theory with particular emphasis on rough differential equations. 
The remainder of the article deals with establishing auxiliary regularity results required for the proof of the main convergence theorem; in particular, the uniform convergence of the rough Hoff process lift and the maximal $p$-variation norm bounds. 

\begin{Notation}
\textup{Throughout the paper, $C_p$ denotes a deterministic constant dependent upon $p$, which may vary from line to line. Denote the set of all partitions of $[s,t]$ by $\Dd_{[s,t]}$. We assume that the continuous semimartingale $X=(X^1,\ldots,X^d)$ has the canonical decomposition $X=M+V$, where $M$ is a local martingale and $V$ is a process of bounded variation. We denote the quadratic variation of $X$ (or more precisely $M$) by setting $\qv{X}:=\{\qv{X^i,X^j} : i,j \in \seq{1,\ldots,d}\} : [0,1]\to\RR^{d\times d}$. Given $p\geq 1$, let $\Mm^p$ denote the class of $L^p$-bounded martingales with the semi-norm 
\[
\norm{M}_p := \abs{\sup_{t\in [0,1]} \norm{M_t-M_0}}_{L^p} = \EE\bra{\sup_{t\in [0,1]} \norm{M_t-M_0}^p}^{1/p}.
\]
Considering one-dimension $M^i$, the Burkholder-Davis-Gundy inequality \cite{burkholder1972integral} states that 
\[
c_p \norm{M^i}_p^p \leq \EE\bra{\qv{M^i}_{0,1}^{p/2}} 
\leq C_p \norm{M^i}_p^p
\]
for some constants $c_p<C_p$ dependent only on $p$. 
Moreover, for simplicity we assume that $X_0=0$. Since the signature of a path is invariant under translation, this assumption comes at no cost of generality.} 
\end{Notation}

\section{Rough path theory and notation}\label{s-rough-path}

In this section we provide a tailored overview of relevant rough path theory and take the opportunity to establish the notation we will need. For a more detailed overview of the theory we direct the reader to
\cite{friz2014course,FV,lejay2003introduction,lyons1994differential,lyons2004differential,lyons2002system} among many others.

{Rough path analysis} was introduced by T. Lyons in the seminal article \cite{lyons1994differential} and provides a method of constructing solutions to differential equations driven by paths that are not of bounded variation but have controlled {roughness}. The measure of this roughness is given by the $p$-variation of the path (see  (\ref{e-p-variation-def}) below). Since this paper works with semimartingales we only need to consider $p\in (2,3)$, but the theory is presented for general $p\geq 1$ because this involves no great additional effort. 

\subsection{Rough path overview}

Since this paper deals with continuous $\RR^d$-valued (and $\RR^{2d}$) paths on $[0,1]$, we restrict ourselves to the finite dimensional case (mainly adopting the notation found in \cite{FV}). Rough paths with values in infinite-dimensional Banach spaces can be defined \cite{lyons2002system}.
Denote the space of continuous paths $x:[0,1]\to\RR^d$ by $C\bra{[0,1],\RR^d}$. We write $x_{s,t}=x_t-x_s$ as a shorthand for the increments of $x\in C\bra{[0,1],\RR^d}$ and let $\norm{x}_\infty$ denote the uniform norm. For $p\geq 1$  define the $p$-variation semi-norm 
\[
	\norm{x}_{p\textup{-var};[0,1]} := \sup_{D=\{t_j\}\in\Dd_{[0,1]}}\bra{ \sum_{t_j\in D} \norm{x_{t_j, t_{j+1}}}^p}^{1/p}.
\]
Let us denote by $C^{p\textup{-var}}\bra{[0,1],\RR^d}$ the linear subspace of $C\bra{[0,1],\RR^d}$ consisting of paths of finite $p$-variation. 
In the case of $x\in C^{p\textup{-var}}\bra{[0,1],\RR^d}$ where  $p\in [1,2)$, the {iterated integrals} of $x$ are canonically defined via Young integration \cite{young1936inequality}. The collection of all these iterated integrals as an object in itself is called the \textit{signature} of the path, given by 
\[
	S(x)_{s,t} := 1 + \sum^\infty_{k=1} \int_{s<t_1<t_2<\ldots<t_k<t} dx_{t_1} \otimes dx_{t_2} \otimes \ldots \otimes dx_{t_k} \in \bigoplus^\infty_{k=0} (\RR^d)^{\otimes k}, 
\]
for all $(s,t) \in \Delta_{[0,1]}:= \seq{(s,t) : 0 \leq s \leq t \leq 1}$. 
With the convention that $(\RR^d)^{\otimes 0} := \RR$, we define the tensor algebras: 
\[
	T^{(\infty)}(\RR^d):= \bigoplus^\infty_{k=0} (\RR^d)^{\otimes k}, \gap\gap\gap T^{(n)}(\RR^d):= \bigoplus^n_{k=0} (\RR^d)^{\otimes k}.
\] 
Thus the signature takes values in $T^{(\infty)}(\RR^d)$. Defining the canonical projection mappings $\pi_n : T^{(\infty)}(\RR^d) \to T^{(n)}(\RR^d)$, we can also consider the \textit{truncated signature}:
\begin{align*}
	S_n(x)_{s,t} :&= \pi_n\bra{S(x)_{s,t}}
	= 1+ \sum^n_{k=1} \int_{s<t_1<t_2<\ldots<t_k<t} dx_{t_1} \otimes dx_{t_2} \otimes \ldots \otimes dx_{t_k} \in T^{(n)}(\RR^d). 
\end{align*}
Hence we can view $S_n$ as a continuous mapping from $\Delta_{[0,1]}$ to $T^{(n)}(\RR^d)$. 
Throughout this paper we will also reserve the notation $\pi_n$ for the canonical projection of $T^{(m)}(\RR^d)$ to $T^{(n)}(\RR^d)$ when $m>n$.

It is a well-known fact that the signature $S_n(x)$ takes values in the step-$n$ free nilpotent Lie group with $d$ generators, which we denote by $G^{(n)}(\RR^d)$. Indeed, defining the free nilpotent step-$N$ Lie algebra $\mathfrak{g}^{(n)}(\RR^d)$ by
\[
	\mathfrak{g}^{(n)}(\RR^d) := \RR^d \oplus [\RR^d,\RR^d] \oplus \ldots \oplus \underbrace{\l[ \RR^d, \l[ \ldots, \l[\RR^d, \RR^d \r] \r]  \r]}_{(n-1) \textup{ brackets}},
\]
and the natural exponential map $\exp_n : T^{(n)}(\RR^d) \to T^{(n)}(\RR^d)$ by 
\[
	\exp_n(a) = 1 + \sum^n_{k=1} \frac{a^{\otimes k}}{k!}, 
\]
we define $G^{(n)}(\RR^d) := \exp_n\bra{\mathfrak{g}^{(n)}(\RR^d)}$. 
The following characterization establishes the well-known fact (a proof can be found in \cite[Theorem 7.30]{FV}). 

\begin{Theorem}[\textbf{Chow-Rashevskii Theorem}]
	We have
	\[	
		G^{(n)}(\RR^d) = \seq{ S_n(x)_{0,1} : x\in C^{1\textup{-var}}\bra{[0,1],\RR^d}}. 
	\]	
\end{Theorem}

More generally, given $p\geq 1$, we set $\floor{p}$ to be the smallest integer less than or equal to $p$  and consider the set of group-valued paths 
\[
	\RP{x}_t = \bra{1,\RP{x}^1_t,\ldots,\RP{x}^{\floor{p}}_t} \in G^{(\floor{p})}(\RR^d). 
\]
Importantly, the group structure provides a natural non-commutative notion of increment: $\RP{x}_{s,t} := \RP{x}_{s}^{-1} \otimes \RP{x}_{t}$.
This multiplication operation is well-defined by Chen's Theorem \cite[Theorem 2.9]{lyons2004differential}. 

\begin{Example}\label{ex-path}
	Take a path $x\in C^{1\textup{-var}}\bra{[0,1],\RR^d}$ and set $\RP{x}:=S_2(x)\in C\bra{[0,1],G^{(2)}(\RR^d)}$ for its level-$2$ signature, which we call the ({level-}$2$) {rough path lift} of $x$. 
	From the definition of the exponential map it is immediate that
		\[
		\RP{x}_{s,t}= 1+ x_{s,t} + \frac{1}{2}x_{s,t}\otimes x_{s,t} + A_{s,t} =\exp_2\bra{x_{s,t}+A_{s,t}},
	\]
	where $A:\Delta_{[0,1]}\to\l[\RR^d,\RR^d\r]$ is the L\'{e}vy area process of $x$. Note that we have used the truncated exponential notation defined above.  
\end{Example}

We can describe the set of \textit{norms} on $G^{(\floor{p})}(\RR^d)$ which are homogeneous with respect to the natural dilation operation on the tensor algebra (see \cite{FV} for more definitions and details). 
The subset of these so-called homogeneous norms which are symmetric and sub-additive gives rise to genuine metrics on $G^{(\floor{p})}(\RR^d)$. We work with the Carnot-Carath\'{e}odory norm on $G^{(n)}(\RR^d)$ given by
\[
	\norm{g}_C := \inf\seq{ \int^1_0 \abs{d\gamma} : \gamma \in C^{1\textup{-var}}\bra{[0,1],\RR^d} \text{ and } S_n(\gamma)_{0,1}=g},
\]
which is well-defined  by the Chow-Rashevskii theorem.
This in turn gives rise to the notion of a homogeneous metric on $G^{(n)}(\RR^d)$:
\[
	d\bra{g,h} = \norm{g^{-1}\otimes h}_C. 
\] 
In fact, recalling the path $x$ from Example \ref{ex-path}, it can be shown that there exists a constants $c,C$ (dependent only on $d$) such that 
\begin{equation}\label{e-area-cc}
c \norm{S_2(x)_{s,t}}_C \leq \norm{x_{s,t}} \vee \sqrt{\norm{A_{s,t}}_{(\RR^d)^{\otimes 2}}} \leq C \norm{S_2(x)_{s,t}}_C.
\end{equation}
We will use this equivalence throughout the paper. 
Moreover, this metric gives rise to the following $p$-variation and $\alpha$-H\"{o}lder metrics on the set of $G^{(\floor{p})}(\RR^d)$-valued paths:
\begin{align}
	\dvar{p}\bra{\RP{x},\RP{y}} :&=  \sup_{D=\{t_j\}\in \Dd_{[0,1]}} \bra{\sum_{t_j \in D} d\bra{\RP{x}_{t_j,t_{j+1}}, \RP{y}_{t_j,t_{j+1}}}^p}^{1/p}\label{e-p-variation-def},\\
	\dhol{1/p}\bra{\RP{x},\RP{y}} :&= \sup_{0\leq s < t \leq 1} \frac{d\bra{\RP{x}_{s,t}, \RP{y}_{s,t}}}{\abs{t-s}^{1/p}}\notag.
\end{align}
We also define the metrics 
\[
	d_{0\textup{;}[0,1]}(\RP{x},\RP{y}) := \sup_{0 \leq s < t \leq 1} d\bra{\RP{x}_{s,t}, \RP{y}_{s,t}}, \gap\gap\gap d_{\infty\textup{;}[0,1]}(\RP{x},\RP{y}) := \sup_{t\in [0,1]} d\bra{\RP{x}_{t},\RP{y}_t}. 
\]
If no confusion may arise we will often drop the $[0,1]$ appearance in these metrics. 
The norms induced by these metrics are denoted by $\norm{\cdot}_{p\textup{-var;}[0,1]}$ and $\norm{\cdot}_{1/p\textup{-H\"{o}l;}[0,1]}$ respectively. 

 
The space of \textit{weakly geometric} $p$-rough paths, denoted by $WG\Omega_p(\RR^d)$, is the  set of paths with values in $G^{(\floor{p})}(\RR^d)$ such that (\ref{e-p-variation-def}) is finite. A refinement of this notion is the space of \textit{geometric} $p$-rough paths, denoted $G\Omega_p(\RR^d)$, which is the closure of 
\[
	\seq{S_{\floor{p}}(x) : x\in C^{1\textup{-var}}\bra{[0,1],\RR^d}},
\]
with respect to the topology induced by the rough path metric $\dvar{p}$. Certainly we have the inclusion $G\Omega_p(\RR^d) \subset WG\Omega_p(\RR^d)$ and it turns out that this inclusion is strict (see \cite[\S3.2.2]{lyons2004differential}). 
	
This paper is concerned with semimartingales, which almost surely have finite $p$-variation for all $p\in (2,3)$ (\cite[Theorem 14.9]{FV}). Thus $\floor{p}=2$ and so we are dealing with $p$-rough paths in the step-$2$ group $G^{(2)}(\RR^d)$. Given a stochastic process $X$, denote its corresponding rough path lift as $\rp{X}$ (as opposed to the rough path lift of $x$ denoted by $\RP{x}$). 

\subsection{Rough differential equations}

In this subsection we introduce rough differential equations. For the more technical topic of rough differential equations {with drift} we refer to the exhaustive \cite[Chapter 12]{FV}. The theory quoted here can be found in greater detail in \cite[Chapter 10]{FV}. 

For now, let $x\in C^{1\textup{-var}}\bra{[0,1],\RR^d}$ and $V=\{V_k\}_{k=1}^d$ be a collection of vector fields $V_k : \RR^q\to\RR^q$.
We denote the solution $y$ of the (controlled) ordinary differential equation (ODE)
\begin{equation}\label{e-ode}
	dy=V(y)\, dx := \sum^d_{k=1} V_k(y)\, dx^k, \gap\gap y_0 \in \RR^q,
\end{equation}
by $\pi_{(V)}\bra{0;y_0,x}$. The notation $\pi_{(V)}\bra{s,y_s;x}$ stands for solutions of (\ref{e-ode}) started at time $s$ from a point $y_s\in\RR^q$. 

\begin{Definition}[\textbf{RDE}]
	Let $\RP{x}\in WG\Omega_p(\RR^d)$ for some $p\geq 1$. We say that $y\in C\bra{[0,1],\RR^q}$ is a solution to the {rough differential equation} (shorthand: {RDE solution}) driven by $\RP{x}$ along the collection of $\RR^q$-vector fields $V=\{V_k\}_{k=1,\ldots,d}$ and started at $y_0$ if there exists a sequence $(x_n)_n$ in $C^{1\textup{-var}}\bra{[0,1],\RR^d}$ such that:
	\begin{enumerate}[label=\textit{(\roman{enumi})}, ref=(\roman{enumi})]\addtolength{\itemsep}{0.4\baselineskip}
		\item $\lim_{n\to\infty} d_{0;[0,1]}\bra{S_{\floor{p}}(x_n),\RP{x}}=0$;
		\item $\sup_n \pnorm{S_{\floor{p}}\bra{x_n}}{p}<\infty$;
	\end{enumerate}
	and ordinary differential equations (ODE) solutions $y_n=\pi_{(V)}\bra{0,y_0;x_n}$ such that 
	\[
		y_n \to y \text{ uniformly on } [0,1] \text{ as } n\to\infty.
	\]
	We denote this situation with the (formal) equation:
	\[
		dy=V(y)\, d\RP{x},\gap\gap y_0\in \RR^q,
	\]
	which we refer to as a rough differential equation. 
\end{Definition}

	The RDE solution map
	\[
		\RP{x} \in C^{p\textup{-var}}\bra{[0,1],G^{(\floor{p})}(\RR^d)} \mapsto y=\pi_{(V)}\bra{0,y_0;\RP{x}} \in C\bra{[0,1],\RR^q}
	\]
	is known in the rough path literature as the {It\^{o} map}.

Considerations of existence and uniqueness of RDEs show that the appropriate way to measure the regularity of the vector field collection $V$ is the notion of $\gamma$-Lipschitzness (denoted $\textup{Lip}^\gamma$) in the sense of E.M. Stein \cite{stein1970singular}. See \cite{FV,lyons2004differential} for a discussion on the contrast with the classical notion of Lipschitzness. The following result is taken from \cite[Theorem 10.26]{FV}.  

\begin{Theorem}
Let $p\geq 1$ and $\RP{x} \in WG\Omega_p(\RR^d)$. If $V=\{V_k\}_{k=1}^d \subset \textup{Lip}^\gamma(\RR^q)$ for some $\gamma>p$ then the RDE 
\[
dy = V(y)\, d\RP{x}
\]
has a unique solution for each initial point $y_0 \in \RR^q$. 
\end{Theorem}

\section{Rough Hoff process convergence}\label{s-main}

This section presents the first main result of the paper. As in Section \ref{s-hoff-define}, we sample a continuous semimartingale $X$ at a sequence of refining partitions $\{D_n\}$ and construct the corresponding sequence of Hoff processes $\{X^{D_n}\}$.
In Theorem \ref{t-hoff-process-convergence} we establish the convergence of the sequence of rough path lifts $\rp{X} ^{D_n}$ of $X^{D_n}$ to a rough path limit under the $p$-variation rough path topology. 

\subsection{Parameterization considerations}

Although $p$-variation estimates and the limit of the Hoff processes are unchanged by reparameterization, the relative parameterization of the Hoff process with respect to the original signal is important when making precise statements about limits and convergence. Certainly the original signal is not invariant to reparameterization. To elaborate further, suppose we have a partition $D_n:=\{t^n_i\}_{i=0}^n\subset[0,1]$ and a sampled signal $X:[0,1]\to\RR^d$. The corresponding Hoff process $X^{D_n}=(X^{D_n,b},X^{D_n,f})$ is not parameterized with respect to the real time of the signal but rather an artificial time; for example, we cannot say how close $X^{D_n,b}_{t^n_k}$ or $X^{D_n,f}_{t^n_k}$ are to the sampled signal point $X_{t^n_k}$ for each $t^n_k\in D_n$. Therefore we need a precise method of comparing the Hoff process to the original signal. To this end, for each partition $D_n$ we define the corresponding time change $\tau_n:[0,1]\to[0,1]$ between artificial and real time by setting $\tau_n\bra{\frac{2k}{2n}}=t^n_k$ and $\tau_n\bra{\frac{2k+1}{2n}}=t^n_{k+1}$,  with piecewise linear interpolation in between. We provide a simple example of the comparison between the original signal and the Hoff process provided by the time change $\tau_n$ in Figure \ref{f-tau-time}, (where we assume $d=1$ for simplicity).  


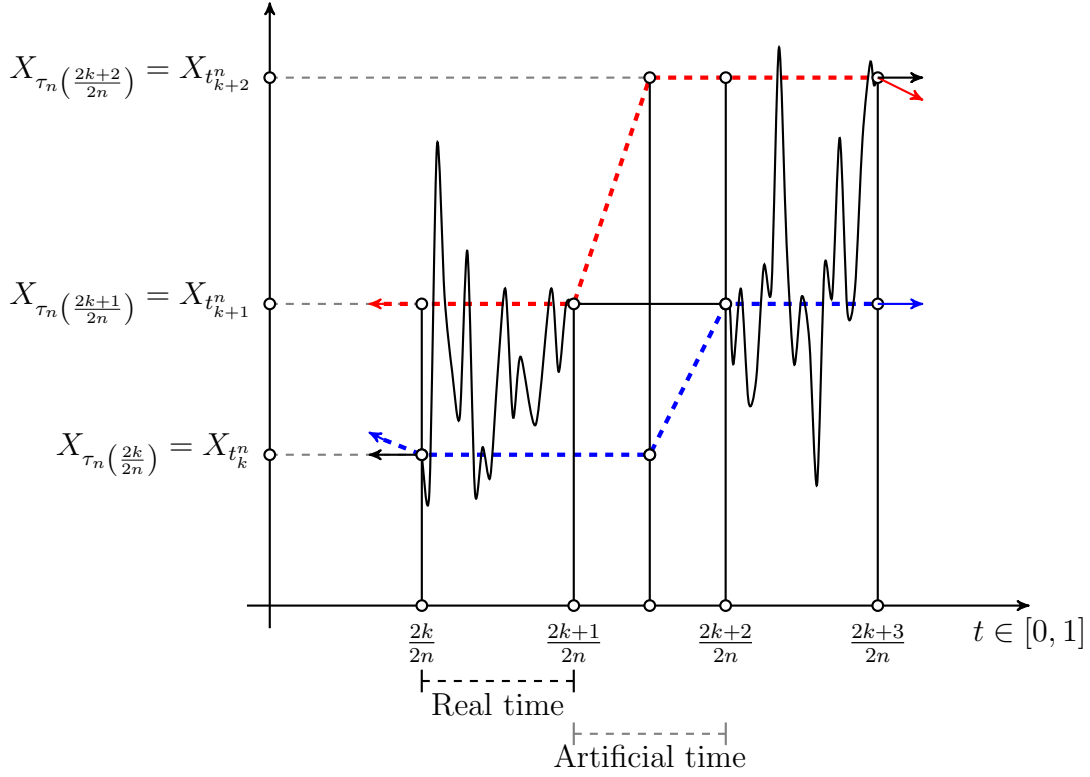
\begin{figure}\label{f-tau-time}
\centering
\usetikzlibrary{arrows,intersections}
\begin{tikzpicture}[
    thick,
    >=stealth',
    dot/.style = {
      draw,
      fill = white,
      circle,
      inner sep = 0pt,
      minimum size = 4pt
    }
  ]
  \coordinate (O) at (0,0);
  \draw[->] (-0.3,0) -- (10,0) coordinate[label = {below:$t\in [0,1]$}] (xmax);
  \draw[->] (0,-0.3) -- (0,8) coordinate[label = {right:$$}] (ymax);
    
  \path[name path=lag] (2,2) -- (5,2) -- (6,4) -- (6.5,4);
  \path[name path=lead] (1,2) -- (2,2) -- (2.5,3) -- (3,3);
  \path[name path=signal] plot[smooth] coordinates 
		{(2,2) (2.1, 1.5) (2.2,6.1) (2.3,4.1) (2.4,3.1) (2.5,2.5) (2.6,4.2) (2.7, 1.5) (2.8,2.1) (2.9, 1.7) (3.0, 3.1) (3.1,4.2) (3.2, 2.1) (3.3, 4.2) (3.5, 2.4) (3.7, 4.2) (3.8, 3.1) (3.9,4) (4,4) (4,4) (4,4) (4,4) (4.1,4) (4.1,4) (5,4) (6,4) (6,4) (6,4) (6.05,4) (6.1,3.2) (6.2,2.7) (6.3,4.5)};
  
  \scope[name intersections = {of = lag and signal, name=X}]
  	 \draw[gray!100, dashed] (0,2) -- (2,2);
	 \draw[gray!100, dashed] (0,4) -- (2,4);
	 \draw[gray!100, dashed] (0,7) -- (5,7);

	\draw[ultra thick, blue,dashed] (1.5,2.2) -- (2,2) -- (5,2) -- (6,4) -- (8,4);
	\draw[ultra thick, red,dashed] (1.5, 4) -- (2,4) -- (4,4) -- (5,7) -- (8,7);
	\draw plot[smooth] coordinates 
		{(1.5, 2) (1.9,2) (2,2) (2,2) (2,2) (2.1, 1.5) (2.2,6.1) (2.3,4.1) (2.4,3.1) (2.5,2.5) (2.6,4.7) (2.7, 1.5) (2.8,2.1) (2.9, 1.7) (3.0, 3.1) (3.1,4.2) (3.2, 2.5) (3.3, 3.3) (3.5, 2.4) (3.7, 4.2) (3.8, 3.1) (3.9,4) (4,4) (4,4) (4,4) (4,4) (4.1,4) (4.1,4) (5,4) (6,4) (6,4) (6,4) (6.05,4) (6.1,3.2) (6.2,4.2) (6.3,2.7) (6.4, 3.0) (6.5,4.5) (6.6,4.2) (6.7, 7.4) (6.8, 4.8) (6.9,3.2) (7.0,4.1) (7.1,3.5) (7.2, 1.6) (7.3, 4.5) (7.4,4.1) (7.5, 6.2) (7.6, 3.9) (7.7, 4.2) (7.8, 6.2) (7.9, 7.2) (7.95,6.9) (8,7) (8,7) (8,7)};
	\draw (X-1) node[dot, label = {above:$$}] (2,5) {} -- node[left]
      {$$} (X-1 |- O) node[dot, label = {below:$\frac{2k}{2n}$}] {};
      	\draw (X-7) node[dot, label = {left:$$}] (2,5) {} -- node[left]
      {$$} (X-7 |- O) node[dot, label = {below:$\frac{2k+2}{2n}$}] {};
      
  \endscope
  
  \scope
  \draw (2,4) node[dot, label={below:$$}] (2,5) {} -- node[left] {$$}  (2,2) node[dot] {};
  \draw (4,0) node[dot, label={below:$\frac{2k+1}{2n}$}] (2,5) {} --  node[left] {$$} (4,4) node[dot] {};
  \draw (6,7) node[dot, label={below:$$}] (2,5) {} --  node[left] {$$} (6,4) node[dot] {};
  \draw (5,0) node[dot, label={above:$$}] (2,5) {} -- node[left] {$$} (5,2) node[dot] {};
  \draw (5,7) node[dot, label={above:$$}] (2,5) {} -- node[left] {$$} (5,2) node[dot] {$$};
  \draw (0,2) node[dot, label={left:$X_{\tau_n\bra{\frac{2k}{2n}}}=X_{t^n_k}$}] {};
  \draw (0,4) node[dot, label={left:$X_{\tau_n\bra{\frac{2k+1}{2n}}}=X_{t^n_{k+1}}$}] {};
  \draw (0,7) node[dot, label={left:$X_{\tau_n\bra{\frac{2k+2}{2n}}}=X_{t^n_{k+2}}$}] {};
  
  \draw (8,7) node[dot] {} (8,7) {} -- (8,4);
  \draw (8,4) node[dot] {} (8,4) {} -- (8,0);
  \draw (8,0) node[dot, label={below:$\frac{2k+3}{2n}$}] {};
  
  \draw[black,dashed] (2, -1) -- (4,-1);
  \draw[black] (2,-0.85) -- (2,-1.15);
  \draw[black] (4,-0.85) -- (4,-1.15);
  \draw[gray,dashed] (4,-1.7) -- (6,-1.7);
  \draw[gray] (4,-1.55) -- (4,-1.85);
  \draw[gray] (6,-1.55) -- (6,-1.85);
  
  \path  (2,-1) -- node[below] {Real time} (4,-1);
  \path  (4,-1.7) -- node[below] {Artificial time} (6,-1.7);
  
  \draw[red,->] (8,7) -- (8.6,6.7);
  \draw[red,->] (1.5,4) -- (1.3,4);
  \draw[blue,->] (1.5,2.2) -- (1.3,2.3);
  \draw[blue,->] (8,4) -- (8.6,4);
  \draw[black,->] (1.5,2) -- (1.3,2);
  \draw[black,->] (8,7) -- (8.6,7);
  
  
  \endscope
  	
\end{tikzpicture}
\caption{Parameterization of the Hoff process $X^{D_n}=(X^{D_n,b},X^{D_n,f})$ with respect to the original signal $(X_{\tau_n(t)})_{t\in [0,1]}$ (in black), now in terms of $\tau_n$. The lag and lead paths are in dashed blue and red lines respectively.}
\end{figure}

Since $\tau_n$ is an increasing (continuous) function, we have no trouble in defining its generalized inverse as $\tau_n^{-1}(u) := \inf\seq{t\geq 0: \tau_n(t)\geq u}$.
As intended, this time change allows us to relate the Hoff process back to the original signal:
\[
X^{D_n}_{\frac{2k+1}{2n}}  = X^{D_n}_{\tau_n^{-1}(t^n_k)} = (X_{t^n_k},X_{t^n_{k+2}}) \,\,\text{ and }\,\, X_{\tau_n\bra{\frac{2k}{2n}}} = X_{t^n_k}, \gap X_{\tau_n\bra{\frac{2k+1}{2n}}}=X_{t^n_{k+1}}.
\]
A similar parameterization using artificial time was previously defined in William's paper \cite{williams2001path} in the context of extending rough path theory to discontinuous L\'{e}vy processes (cf. \cite[Definition 1.3]{williams2001path}). 

\subsection{Rough Hoff convergence}

We present both of the two main theorems of this paper in terms of these time changes $\tau_n$. 

\begin{Theorem}[\textbf{Rough Hoff process convergence}]\label{t-hoff-process-convergence}
	Let $X=M+V:[0,1]\to\RR^d$ be a continuous semimartingale with L\'{e}vy area $A$, and define the associated Hoff processes $X^{D_n}$ corresponding to the collection of sampling times $\{D_n\}_{n=1}^\infty\subset \Dd_{[0,1]}$. Then for all $p>2$ we have
	\[
		\dvar{p}\bra{\rp{X}^{D_n}, \rp{X}^\infty_{\tau_n(\cdot)}} \to 0 \text{ in probability as } \abs{D_n}\to 0,
	\]
	where $\rp{X}^\infty \in C\bra{[0,1],G^{(2)}(\RR^{2d})}$ is defined by
\[
	\rp{X}^\infty_{s,t} = \exp_2\bra{
	\left(\begin{array}{c}
X_{s,t}\\
X_{s,t}
\end{array}\right) +  
	\left(\begin{array}{cc}
A_{s,t} & A_{s,t}-\frac{1}{2}\qv{X}_{s,t}\\
A_{s,t}+\frac{1}{2}\qv{X}_{s,t} & A_{s,t}
\end{array}\right)
	}.
\]
	If in addition $V$ is bounded in $L^r$ for some $r\geq 1$ and $M\in\Mm^p$, then the convergence also holds in the $L^{p\wedge r}$-norm.  
\end{Theorem}
\begin{proof}
Throughout the proof we refer to technical results in Sections \ref{s-uniform} and \ref{s-maximal}. 
	First we prove the theorem for $V=0$. Assuming $M\in\Mm^p \subseteq \Mm^q$, where $q\in (2,p)$, it follows from Proposition \ref{p-maximal-p-variation-bound} that the rough path collection $\seq{\rp{M}^{D_n}}_{n=1}^\infty$ satisfies uniform $q$-variation bounds in $L^q$:
	\begin{equation}\label{e-main-hoff-maximal}
	\sup_{n} \abs{ \norm{\rp{M}^{D_n}}_{q\textup{-var;}[0,1] }}_{L^q} = C_q < \infty.
	\end{equation}
	Moreover, since each $\rp{M}_{\tau_n(\cdot)}$ is simply a reparameterization of the rough path lift of the original martingale $S_2(\rp{M})$, by \cite[Proposition 14.9]{FV} we have 
	\[
	\sup_n\abs{ \norm{ \rp{M}_{\tau_n(\cdot)}}_{q\textup{-var;}[0,1]}}_{L^q} =\abs{ \norm{ \rp{M}}_{q\textup{-var;}[0,1]}}_{L^q} = C_q < \infty.
	\]	
	Recalling the interpolation result of \cite[Proposition 12]{friz2006note}, for $2<q<p$ it follows that
	\begin{align*}
	d_{p\textup{-var};[0,1]}\bra{\rp{M}^{D_n}, \rp{M}^\infty_{\tau_n(\cdot)}}
	&\leq d_{\infty;[0,1]}\bra{\rp{M}^{D_n}, \rp{M}^\infty_{\tau_n(\cdot)}}^{1-q/p}\cdot d_{q\textup{-var;}[0,1]}\bra{\rp{M}^{D_n}, \rp{M}^\infty_{\tau_n(\cdot)}}^{q/p}\\
	&\leq d_{\infty;[0,1]}\bra{\rp{M}^{D_n}, \rp{M}^\infty_{\tau_n(\cdot)}}^{1-q/p}\cdot \bra{ \norm{\rp{M}^{D_n}}_{q\textup{-var;}[0,1] } + \norm{\rp{M}_{\tau_n(\cdot)}}_{q\textup{-var;}[0,1] }}^{q/p}.
	\end{align*}
	Using H\"{o}lder's inequality with $1/r+1/s=1$ satisfying $r=\frac{p}{p-q}>1$ and $s=\frac{p}{q}\in (1,3/2)$, we combine the maximal $q$-variation bounds (\ref{e-main-hoff-maximal}) with the interpolation inequality to find that 
	\[
	\abs{ 
	d_{p\textup{-var};[0,1]}\bra{\rp{M}^{D_n}, \rp{M}^\infty_{\tau_n(\cdot)}}
	}_{L^p} \leq C_{p,q} \abs{d_{\infty;[0,1]}\bra{\rp{M}^{D_n}, \rp{M}^\infty_{\tau_n(\cdot)}}}_{L^p}^{1/(p-q)}.
	\]
	Thus the second claimed convergence in the $L^p$-norm follows from the uniform convergence result of Proposition \ref{p-uniform-convergence}. In the case of $M\notin \Mm^p$, we obtain convergence in probability by a simple localization argument (as done previously in \cite[Theorem 14.16]{FV}). 

	We now turn to the general case where $V\neq 0$ and define $\bar{V} :[0,1]\to\RR^{2d}$ by $\bar{V}_t=(V_t,V_t)$. In contrast to the martingale rough path convergence, the corresponding convergence proof for the finite variation process is straightforward.

\begin{Lemma}\label{l-fv-convergence}
	For every $\delta>0$,
	\[
		\dvar{(1+\delta)}\bra{V^{D_n},\bar{V}_{\tau_n(\cdot)}} \to 0 \text{ in probability as } \abs{D_n}\to0.
	\]
	If in addition $V$ is bounded in $L^r$ for some $r\geq 1$ then the convergence also holds in $L^r$. 
\end{Lemma}
\begin{proof}
	The first convergence follows readily from \cite[Proposition 1.28]{FV} and interpolation (see \cite[\S14.1]{FV} for details). The stronger convergence in the $L^r$-norm is a consequence of the uniform $L^r$-bound assumption (\cite[Theorem 4.14]{chung2001course}). 
\end{proof}
	
	Returning to the proof of Theorem \ref{t-hoff-process-convergence}, it is a straightforward exercise in algebra to confirm that $\rp{X}^\infty_{s,t} = T_{\bar{V}}\bra{\rp{X}^\infty_{s,t}}$, where $T_{\bar{V}}$ is the standard rough path translation operator defined in \cite[\S9.4.6]{FV}. The theorem then follows immediately from the continuity of the translation operator (\cite[Corollary 9.35]{FV}). 
\end{proof}	

\begin{Remark}
The quadratic variation $\qv{X}=\qv{M}$ of the local martingale component $M$ naturally appears as we consider the limiting area of the rough path lift $\rp{X}^D$. This non-linear phenomena was first observed by Hoff in the context of moving  Brownian frames \cite{hoff2005brownian}. 
\end{Remark}

\begin{Remark}
In the special case of the evenly spaced sampling sequence, that is $D_n=\{t^n_k:=\frac{k}{n}\}_{k=0}^n$, we have $\tau_n\bra{\frac{k}{n}}=\frac{k}{n}$ and $\tau_n\bra{\frac{2k+1}{2n}}=\frac{k+1}{n}$. It follows immediately that we can drop the time change $\tau_n(\cdot)$ in Theorem \ref{t-hoff-process-convergence} to arrive at the neater convergence result: 
\[
d_{p\textup{-var;}[0,1]}\bra{\XX^{D_n}, \XX^\infty} \to 0 \text{ as } n\to\infty,
\]
either in probability, or in $L^{p\wedge q}$, (depending on the regularity imposed on $X$). 
\end{Remark}

\section{Recovery of the It\^{o} integral}\label{s-ito-recover}

We now consider the second main result of the paper.

\begin{Theorem}[\textbf{Recovery of the It\^{o} integral}]\label{t-ito-sde}
		Let $X:[0,1]\to\RR^d$ be a continuous semimartingale and let $\{D_n\}_{n=1}^\infty \subset \Dd_{[0,1]}$ be a sequence of times satisfying $\abs{D_n}\to 0$. 
		Let $f=(f_i)_{i=1}^d$ be a collection of $\textup{Lip}^\gamma(\RR^d,\RR^q)$ mappings where $\gamma > 2$. 
		For each $n$, set $X^n:=X^{D_n}$ to be the associated Hoff process of the event stream $\seq{X_{t_i}}_{t_i\in D_n}$. Lastly, let $Y^n$ be the solution to the random ODE
		\begin{equation}\label{e-delay}
			dY^n_t = f\bra{X_t^{n,b}}\, dX^{n,f}_t = \sum^d_{i=1} f_i\bra{X^{n,b}_t}\, dX^{n,f;i}_t, \gap\gap Y^n_0 = y_0 \in \RR^q,
		\end{equation}
		and set $Y$ to be the standard It\^{o} integral
		\[
			dY_t = f\bra{X_t}\, dX_t = \sum^d_{i=1} f_i\bra{X_t}\, dX^i_t, \gap\gap Y_0 = y_0 \in \RR^q. 
		\]
		Then for all $p\in (2,\gamma)$, 
		\begin{equation}\label{e-ito-convergence-result}
			\dvar{p}\bra{Y^n,Y_{\tau_n(\cdot)}}\to 0 \text{ in probability as } n\to\infty.
		\end{equation}
		If in addition the finite variation process $V$ of $X$ is bounded in $L^q$ for some $q\geq 1$ and $M\in\Mm^p$, then the convergence (\ref{e-ito-convergence-result}) also holds in the $L^{p\wedge q}$-norm. 
\end{Theorem}

Before presenting the proof of Theorem \ref{t-ito-sde} we set up some notation. Denote the rough path lift  of $\bar{X_t}:=(X_t,X_t)$ by $\rp{\bar{X}}$. Thus $\rp{\bar{X}}: \Delta_{[0,1]} \to G^{(2)}(\RR^{2d})$ with 
\begin{align*}
	\rp{X}_{s,t} :&=  1 + \bar{X}_{s,t}
		+ \frac{1}{2} \bar{X}_{s,t} \otimes \bar{X}_{s,t}
+
\left(\begin{array}{cc}
A_{s,t} & A_{s,t}\\
A_{s,t} & A_{s,t}
\end{array}\right),
\end{align*}
where $A : \Delta_{[0,1]} \to \l[\RR^d,\RR^d\r]$ is the L\'{e}vy area process of $X$. 
Importantly we have $\rp{X}^\infty = \rp{\bar{X}}^{\psi}$; that is $\rp{X}^\infty$ is the rough path created from perturbing $\rp{\bar{X}}$ by the map $\psi:\Delta_{[0,1]} \to [\RR^{2d}, \RR^{2d}]$:
\begin{align*}
	\rp{X}^\infty = \rp{\bar{X}}_{s,t} + \psi_{s,t} 
	&= \exp_2\bra{\bar{X}_{s,t} +
	\left(\begin{array}{cc}
A_{s,t} & A_{s,t}\\
A_{s,t} & A_{s,t}
\end{array}\right)
} + \psi_{s,t}\\
&
	= \exp_2\bra{\bar{X}_{s,t} +
	\left(\begin{array}{cc}
A_{s,t} & A_{s,t}\\
A_{s,t} & A_{s,t}
\end{array}\right)
 + \psi_{s,t} }, 
\end{align*}
where the path $\psi$ is defined by
\[
	\psi_{s,t} := \left(\begin{array}{cc}
0 & -\frac{1}{2}\qv{X}_{s,t}\\
\frac{1}{2}\qv{X}_{s,t} & 0
\end{array}\right) \in \l[ \RR^{2d}, \RR^{2d} \r].
\]
The appearance of the covariation term $\psi$ in the limit $\rp{X}^\infty$ of Theorem \ref{t-hoff-process-convergence} allows us to recover the It\^{o} integral from the Stratonovich integral limit. In particular, the drift terms from the Stratonovich correction cancel out. 
We reserve the notation $\rho : \RR^d \oplus \RR^q \to \RR^q$ for the canonical projection. 

\begin{proof}[Proof of Theorem \ref{t-ito-sde}]
	We consider a more complicated SDE on the larger space $\RR^d \oplus \RR^q$ and denote the standard bases of $\RR^d$ and $\RR^q$ by $(\hat{e}_i)_{i=1}^d$ and $(\bar{e}_j)_{j=1}^q$ respectively. Indeed, set $z=(\hat{z}, \bar{z})$ to be a path in $\RR^d \oplus \RR^q$ and define vector fields on $\RR^d \oplus \RR^q$ by  
	\[
		Q_i(z):=(\hat{e}_i,0), \gap\gap W_k(z):= \sum^q_{j=1} (0, \bar{e}_j)f^j_k(\hat{z}),
	\] 
	for $i,k=1,\ldots,d$. Here we have used the standard notation convention: $f_i(x)=(f_i^1(x), \ldots, f_i^q(x)) \in \RR^q$ for $x\in \RR^d$.
	Let $z^n=(\hat{z}^{n,1},\ldots, \hat{z}^{n,d}, \bar{z}^{n,1}, \ldots, \bar{z}^{n,q})$ be the solution to the SDE
	\[
		dz^n_t= \sum^d_{i=1}\bra{ Q_i(z^n_t)\, dX^{n,b;i}_t + W_i(z^n_t)\, dX^{n,f;i}_t}, \gap\gap z^n_0 = (0,y_0) \in \RR^d \oplus \RR^q. 
	\]
	Since the axis path $X^n$ is piecewise linear, (thus of finite variation), we could equivalently formulate the above SDE into its Stratonovich form without incurring a corrective drift term. 
	It follows that the projection $\bar{z}^n := \rho(z^n)$ of $z^n$ onto $\RR^q$ satisfies the SDE
	\[
		d\bar{z}^{n,j}_t = \sum^d_{i=1} W_i^j\bra{z^n_t}\, dX^{n,f;i}_t = \sum^d_{i=1} f^j_i\bra{\hat{z}^n_t}\, dX^{n,f;i}_t. 
	\]
	Moreover, Theorem 17.3 of \cite{FV} tells us that $z^n=\pi_{(Q,W)}\bra{0, (0,y_0); \rp{X}^n}$ almost surely. 	
	
	A simple computation confirms that the only non-zero Lie brackets between the vector fields $Q_{i_1}, Q_{i_2}, W_{k_1}$ and $W_{k_2}$ are of the form 
	\[
		\l[Q_i, W_k\r](\hat{z}) = Q_iW_k(\hat{z}) = \sum^e_{j=1} (0,\bar{e}_j)  \frac{\partial f^j_k}{\partial x_i}(\hat{z}),
	\]
	for all $i,k=1,\ldots, d$. We denote this collection of vector fields by 
	\[
		L:=\bra{Q_iW_k : i, k \in \seq{1,\ldots, d}}.
	\]
 	Next consider the SDE:
 	\begin{align*}
 		dz_t = (Q+W)(z_t)\, dX_t &= (Q+W)(z_t) \circ dX_t - \frac{1}{2}\sum^d_{i,k=1} {Q_iW_k(z_t)\, d\ip{X^i}{X^k}_t}\\
 	&= (Q+W)(z_t)\circ dX_t - \frac{1}{2}\sum^d_{i,k=1} \l[ Q_i, W_k\r](z_t)\, d\ip{X^i}{X^k}_t, 
 	\end{align*}
 	with initial condition $z_0 = (0, y_0) \in \RR^d \oplus \RR^q$. 
 	Again appealing to Theorem 17.3 of \cite{FV}, it follows that almost surely
 	\[
 		z=\pi_{(Q,W;L)}\bra{0,(0,y_0); (\rp{X}, \psi)}.
	\]
	 As noted above, $\bar{\rp{X}}^\psi=\rp{X}^\infty$, and so the perturbation theorem of \cite[Theorem 12.14]{FV} guarantees the equality:
 	\begin{align*}
 		z&=\pi_{(Q,W;L)}\bra{0, (0,y_0); (\rp{X}, \psi)}
		= \pi_{(Q,W)}\bra{0, (0,y_0); \bar{\rp{X}}^\psi}
		= \pi_{(Q,W)}\bra{0, (0,y_0); \rp{X}^\infty}.
 	\end{align*}
 	In light of the convergence result of Theorem \ref{t-hoff-process-convergence}, we conclude that
 	\[
 		\forall p> 2,\,\,\, \dvar{p}\bra{z^n,z_{\tau_n(\cdot)}} \to 0 \text{ in probability (or in } L^{p\wedge q} \text{) as }  n\to\infty.
 	\] 
 	Indeed, as noted in the above remark, almost sure equality and convergence in probability (and $L^p$) are all preserved under continuous mappings. Since the It\^{o} map is continuous with respect to the driving rough path signal under the $p$-variation topology, the proof is then finished by noting that $Y^n=\rho(z^n)$ and $Y=\rho(z)$. 
\end{proof}

The remainder of the paper is devoted to establishing the technical results referred to in the proof of Theorem \ref{t-hoff-process-convergence}.

\section{Uniform convergence of the Hoff process}\label{s-uniform}

In this section we establish the following uniform convergence in $L^p$ of the Hoff process in the case that $X=M\in \Mm^p$. 

\begin{Proposition}\label{p-uniform-convergence}
Let $M\in\Mm^p$ for some $p>2$ and let $\{D_n\}_{n=1}^\infty \subset \Dd_{[0,1]}$ be a sequence of partitions with $\abs{D_n}\to 0$. Then
\[
\abs{d_{\infty;[0,1]}\bra{\rp{M}^{D_n}, \rp{M}_{\tau_n(\cdot)}}}_{L^p} \to 0 \textup{ as } \abs{D_n}\to 0.
\]
\end{Proposition}

Before presenting a proof, we introduce two useful technical lemmas concerning $\Mm^p$-martingales. We make frequent use of the famous Burkholder-Davis-Gundy (BDG) inequality (\cite[Theorem 42.1]{rogers2000diffusions}).

\begin{Lemma}\label{l-irritating}
Suppose $X:[0,1]\to\RR$ is a $\Mm^p$-martingale for some $p>2$. Then there exists a constant $C_p>0$ such that for all $D=\{t_k\}_{k=0}^n \subset [0,1]$, we have
\[
\EE\bra{\sup_{t_k\in D} \sup_{t\in [t_k,t_{k+1}]} \abs{X_t-X_{t_{k}}}^p}
\leq C_p \norm{X}_p^2  \EE\bra{\sup_{t_k\in D} \qv{X}^{p/2}_{t_k,t_{k+1}}}^{\frac{p-2}{p}}.
\]
\end{Lemma}
\begin{proof}
We first apply the BDG inequality to deduce that
\begin{align*}
\EE\bra{\sup_{t_k\in D} \sup_{t\in [t_k,t_{k+1}]} \abs{X_t-X_{t_k}}^p}
&\leq \sum_{t_k\in D} \EE\bra{\sup_{t\in [t_k,t_{k+1}]} \abs{X_t-X_{t_k}}^p}
\leq C_p \sum_{t_k\in D} \EE\bra{\l<X\r>_{t_k,t_{k+1}}^{p/2}}. 
\end{align*}
Then we apply the H\"{o}lder inequality with $q=p/2$ and $r=p/(p-2)$, (so that $q^{-1}+r^{-1}=1$), to find
\begin{align*}
\sum_{t_k\in D} \EE\bra{\l<X\r>_{t_k,t_{k+1}}^{p/2}}&=\EE\bra{ \sum_{t_k\in D} \l<X\r>_{t_k,t_{k+1}}\cdot \sup_{t_k\in D} \l<X\r>_{t_k,t_{k+1}}^{p/2-1}}\\
&= \EE\bra{  \l<X\r>_{0,1}\cdot \sup_{t_k\in D} \l<X\r>_{t_k,t_{k+1}}^{p/2-1}}\\
&\leq \abs{\qv{X}_{0,1}}_{L^r} \cdot \abs{\sup_{t_k\in D} \qv{X}^{p/2-1}_{t_k,t_{k+1}}}_{L^q}\\
&= \EE\bra{\qv{X}^{p/2}_{0,1}}^{2/p} \EE\bra{\sup_{t_k\in D} \qv{X}^{p/2}_{t_k,t_{k+1}}}^{\frac{p-2}{p}}
\leq C_p \norm{X}_p^2  \EE\bra{\sup_{t_k\in D} \qv{X}^{p/2}_{t_k,t_{k+1}}}^{\frac{p-2}{p}}.
\end{align*}
The proof is complete. 
\end{proof}

\begin{Lemma}\label{l-useful-2}
Suppose $X,Y:[0,1]\to\RR$ are $\Mm^p$-martingales for some $p>2$. Then there exists a constant $C_p>0$ such that for $D=\{t_k\}_{k=0}^n \subset [0,1]$, the following inequalities hold:
\begin{align*}
\EE\bra{\sup_{t_k\in D} \abs{\sum_{t_i\leq t_k} X_{t_{i-1}}(Y_{t_i}-Y_{t_{i-1}}) - \int^{t_k}_0 X_r\, dY_r}^{p/2}}&\leq 
C_p \norm{Y}^{p/2}_p \norm{X}_p \EE\bra{\sup_{t_k\in D} \l<X\r>_{t_k,t_{k+1}}^{p/2}}^{\frac{p-2}{2p}};\\
\EE\bra{\sup_{t_k\in D} \abs{\sum_{t_i \leq t_k} \bra{X_{t_{i+1}}-X_{t_i}}\bra{Y_{t_{i+2}}-Y_{t_{i+1}}}}^{p/2}} &\leq C_p\norm{X}^{p/2}_p \norm{Y}_p \EE\bra{\sup_{t_k\in D} \l<Y\r>_{t_k,t_{k+1}}^{p/2}}^{\frac{p-2}{2p}}.
\end{align*}
\end{Lemma}
\begin{proof}
Define the previsible process $Z_r:= \sum_{0\neq t_i \in D} X_{t_{i-1}}\id\{r\in (t_i,t_{i+1}]\}$. By the Kunita-Watanabe identity, along with the BDG and Cauchy-Schwarz inequalities, we have 
\begin{align*}
&\EE\bra{\sup_{t_k\in D} \abs{\sum_{t_i\leq t_k} X_{t_{i-1}}(Y_{t_i}-Y_{t_{i-1}}) - \int^{t_k}_0 X_r\, dY_r}^{p/2}}\\
&=
\EE\bra{\sup_{t_k\in D} \abs{\int^{t_k}_0 (Z_r-X_r)\, dY_r}^{p/2}}\\
&\leq C_p\EE\bra{ \abs{\int^1_0 \abs{Z_r-X_r}^2\, d\l<Y\r>_r}^{p/4}}\\
&
\leq C_p\EE\bra{\sup_{r\in [0,1]} \abs{Z_r-X_r}^{p/2} \l<Y\r>_{0,1}^{p/4}}\\
&
\leq C_p\sqrt{ \EE\bra{\sup_{r\in [0,1]} \abs{Z_r-X_r}^{p}} \EE\bra{\l<Y\r>^{p/2}_{0,1}}}
\leq C_p \norm{Y}_p^{p/2} \sqrt{ \EE\bra{\sup_{r\in [0,1]} \abs{Z_r-X_r}^{p}}}.
\end{align*}
Note that
\begin{align*}
\EE\bra{\sup_{r\in [0,1]} \abs{Z_r-X_r}^{p}}
&= \EE\bra{\sup_{t_i \in D} \sup_{t\in [t_{i+1},t_{i+2}]} \abs{X_t-X_{t_i}}^p}
\leq C_p\EE\bra{\sup_{t_i\in D} \sup_{t\in [t_i,t_{i+1}]} \abs{X_t-X_{t_i}}^p}.
\end{align*}
Thus Lemma \ref{l-irritating} gives:
\begin{align*}
\EE\bra{\sup_{r\in [0,1]} \abs{Z_r-X_r}^{p}}
&\leq C_p \norm{X}_p^2  \EE\bra{\sup_{t_k\in D} \qv{X}^{p/2}_{t_k,t_{k+1}}}^{\frac{p-2}{p}},
\end{align*}
and the first claimed inequality follows. 
The second inequality has an almost identical proof; we simply use $Z_r:= \sum_{0\neq t_i \in D} Y_{t_{i-1}}\id\{r\in (t_i,t_{i+1)}]\}$ instead. 
\end{proof}

\begin{proof}[Proof of Proposition \ref{p-uniform-convergence}]
For notational convenience we drop the usual $n$ superscript in $t^n_k$ and just write $D_n=\{t_k\}_{k=0}^n$. 
We first consider the increment level. For $\theta\in\{b,f\}$, using Lemma \ref{l-irritating} gives
\begin{align*}
\EE\bra{\sup_{s\neq t} \abs{M^{D_n,\theta;i}_{s,t} - M^i_{\tau_n(s),\tau_n(t)}}^p}
&\leq C_p \norm{M^i}_p^2 \EE\bra{\sup_{t_k\in D_n} \l<M^i\r>_{t_k,t_{k+1}}^{p/2}}^{\frac{p-2}{p}}.
\end{align*}
Since $M\in\Mm^p$, we use the Dominated Convergence Theorem (DCT) to deduce that 
\[
\EE\bra{\sup_{s\neq t} \abs{M^{D_n,\theta;i}_{s,t} - M^i_{\tau_n(s),\tau_n(t)}}^p}
\to 0 \textup{ as } \abs{D_n}\to 0.
\]

We now turn our attention to the second level.
Let us denote the L\'{e}vy area process associated with $M^{D_n}$ by 
\begin{align}
A^{D_n,\theta,\lambda;i,j}_{s,t} 
:&= \frac{1}{2}\int^t_s \bra{M^{D_n,\theta;i}_{s,r}\, dM^{D_n,\lambda;j}_{r} - M^{D_n,\lambda;j}_{s,r}\, dM^{D_n,\theta;i}_r}\notag\\ 
&= \int^t_s M^{D_n,\theta;i}_r\, dM^{D_n,\lambda;j}_r - \frac{1}{2} M^{D_n,\theta;i}_{s,t} M^{D_n,\lambda;j}_{s,t},\label{e-m-levy-area}
\end{align}
where $\theta,\lambda\in\{b,f\}$ and $i,j\in\{1,\ldots,d\}$. Since $M^{D_n}$ is piecewise linear, (and thus has zero quadratic variation), we are able to switch between the It\^{o} and Stratonovich integrals in (\ref{e-m-levy-area}) without incurring a correction term. 

 Recalling (\ref{e-area-cc}), we have 
\begin{align*}
&d_{\infty;[0,1]}\bra{\rp{M}^{D_n}, \rp{M}_{\tau_n(\cdot)}}
\leq C_d \sup_{s\neq t} \bra{ \norm{M^{D_n}_{s,t} - \bar{M}_{\tau_n(s),\tau_n(t)}}^p + 
 \norm{A^{D_n}_{s,t} - \hat{A}_{\tau_n(s),\tau_n(t)}}^{p/2}},
\end{align*}
where $\hat{A} = \{\hat{A}^{\theta,\lambda;i,j}\}$ is defined by
\[
	   \hat{A}^{\theta,\lambda;i,j}_{s,t} := \begin{cases}
A_{s,t}^{M;i,j} & \;\;\;\;\;\;\;\text{ if } \theta = \lambda;\\
A_{s,t}^{M;i,j}-\frac{1}{2}\ip{M^{i}}{M^{j}}_{s,t} & \;\;\;\;\;\;\;\text{ if } \theta = b, \; \lambda=f;\\
A_{s,t}^{M;i,j} + \frac{1}{2}\ip{M^{i}}{M^{j}}_{s,t} &  \;\;\;\;\;\;\;\text{ if } \theta = f, \; \lambda=b.
\end{cases}
	\]
Thus to complete the proof it suffices to show that 
\begin{equation}\label{e-area-goal-1}
\EE\bra{\sup_{s\neq t} \abs{A^{D_n,\theta,\lambda;i,j}_{s,t} - \hat{A}^{\theta,\lambda;i,j}_{\tau_n(s),\tau_n(t)}}^{p/2}} \to 0 \textup{ as } \abs{D_n} \to 0. 
\end{equation}
		
Given a bounded variation path $x:[0,1]\to\RR^d$ with area process $A:\Delta_{[0,1]}\to [\RR^d,\RR^d]$, we have the decomposition:
\[
A_{s,t} = A_{0,t} - A_{0,s} - \frac{1}{2}[x_{0,s},x_{s,t}] \gap \forall (s,t)\in \Delta_{[0,1]}.
\]
In light of this formula, the fact that straight line segments enclose zero area, and the previously established uniform convergence at level 1, to show (\ref{e-area-goal-1}) it is sufficient to only prove that 
\[
\EE\bra{\sup_{t_k \in D_n} \abs{A^{D_n,\theta,\lambda;i,j}_{0,\frac{k}{n}} - \hat{A}^{\theta,\lambda;i,j}_{0,\tau_n\bra{\frac{k}{n}}}}^{p/2}} \to 0 \textup{ as } \abs{D_n} \to 0.
\]

To this end, we first consider the case of $(\theta,\lambda)=(b,b)$, (the case of $(f,f)$ is almost identical and thus omitted). Setting $X:=M^i$, $Y:=M^j$ allows us to free up $(i,j)$ for index notation. A simple computation gives:
\begin{align*}
A^{D_n,b,b;i,j}_{0,\frac{k}{n}} &= \frac{1}{2}\sum_{i\leq k} \seq{X_{t_{i-1}}(Y_{t_i}-Y_{t_{i-1}}) - Y_{t_{i-1}}(X_{t_i}-X_{t_{i-1}})};\\
\hat{A}^{b,b;i,j}_{0,\tau_n\bra{\frac{k}{n}}} &= A^{M;i,j}_{0,t_k} = \frac{1}{2}\int^{t_k}_0 \bra{X_r\, dY_r - Y_r\, dX_r}.
\end{align*}
It follows that 
\begin{align*}
\EE\bra{\sup_{t_k \in D_n} \abs{A^{D_n,b,b;i,j}_{0,\frac{k}{n}} - \hat{A}^{b,b;i,j}_{0,\tau_n\bra{\frac{k}{n}}}}^{p/2}}
&\leq C_p \EE\bra{\sup_{t_k\in D_n} \abs{\sum_{t_i\leq t_k} X_{t_{i-1}}(Y_{t_i}-Y_{t_{i-1}}) - \int^{t_k}_0 X_r\, dY_r}^{p/2}}\\
& + C_p \EE\bra{\sup_{t_k\in D_n} \abs{\sum_{t_i\leq t_k} Y_{t_{i-1}}(X_{t_i}-X_{t_{i-1}}) - \int^{t_k}_0 Y_r\, dX_r}^{p/2}}. 
\end{align*}
By applying Lemma \ref{l-useful-2} to both terms, we can then use the DCT to find that 
\begin{align*}
&\EE\bra{\sup_{t_k \in D_n} \abs{A^{D_n,b,b;i,j}_{0,\frac{k}{n}} - \hat{A}^{b,b;i,j}_{0,\tau_n\bra{\frac{k}{n}}}}^{p/2}}\\
&\gap\gap\gap\gap \leq  C_{p,M}\bra{ \EE\bra{\sup_{t_k\in D_n} \l<X\r>_{t_k,t_{k+1}}^{p/2}}^{\frac{p-2}{2p}} +  \EE\bra{\sup_{t_k\in D_n} \l<Y\r>_{t_k,t_{k+1}}^{p/2}}^{\frac{p-2}{2p}}}\\
&\gap\gap\gap\gap \to 0 \textup{ as } \abs{D_n}\to 0.
\end{align*} 

Since $A^{D_n,b,f;i,j}=-A^{D_n,f,b;j,i}$, it only remains to prove the claimed convergence for the case of $(\theta,\lambda)=(b,f)$. To this end, we first note that 
\begin{equation}\label{e-bf-area}
A^{D_n,b,f;i,j}_{0,\frac{k}{n}} = \sum_{t_i \leq t_k} X_{t_i}(Y_{t_{i+2}}-Y_{t_{i+1}}) - \frac{1}{2}X_{0,t_k}Y_{t_1,t_{k+2}}.
\end{equation}
We have the decomposition
\[
\sum_{t_i \leq t_k} X_{t_i}(Y_{t_{i+2}}-Y_{t_{i+1}}) = \sum_{t_i\leq t_k} X_{t_{i+1}}(Y_{t_{i+2}}-Y_{t_{i+1}}) - \sum_{t_i\leq t_k} (X_{t_{i+1}}-X_{t_i})(Y_{t_{i+2}}-Y_{t_{i+1}}). 
\]
As before, we use the second inequality of Lemma \ref{l-useful-2} to deal with the second term:
\[
\EE\bra{\sup_{t_k\in D_n} \abs{\sum_{t_i \leq t_k} \bra{X_{t_{i+1}}-X_{t_i}}\bra{Y_{t_{i+2}}-Y_{t_{i+1}}}}^{p/2}} 
\leq C_{p,X,Y} \EE\bra{\sup_{t_k\in D_n} \l<Y\r>_{t_k,t_{k+1}}^{p/2}}^{\frac{p-2}{2p}}. 
\]
The first inequality of the same lemma guarantees that 
\[
\EE\bra{\sup_{t_k\in D_n} \abs{\sum_{t_i\leq t_k} X_{t_{i}}(Y_{t_{i+1}}-Y_{t_{i}}) - \int^{t_k}_0 X_r\, dY_r}^{p/2}} \leq 
C_{p,X,Y} \EE\bra{\sup_{t_k\in D_n} \l<X\r>_{t_k,t_{k+1}}^{p/2}}^{\frac{p-2}{2p}}.
\]
Therefore,
\begin{align}
&\EE\bra{\sup_{t_k \in D_n} \abs{A^{D_n,b,f;i,j}_{0,\frac{k}{n}} - \bra{\int^{t_k}_0 X_r\, dY_r - \frac{1}{2}X_{t_k}Y_{t_k} }}^{p/2}}\notag\\
&\gap\gap\gap\gap \leq C_{p,X,Y}2^{p/2-1} \bra{\EE\bra{\sup_{t_k\in D_n} \l<X\r>_{t_k,t_{k+1}}^{p/2}}^{\frac{p-2}{2p}} +  \EE\bra{\sup_{t_k\in D_n} \l<Y\r>_{t_k,t_{k+1}}^{p/2}}^{\frac{p-2}{2p}}}.\label{e-therefore-2}
\end{align}
Rearranging via the integration by parts formula for continuous (semi-)martingales $M,N$:
\[
d(MN)_t = M_tdN_t + N_tdM_t + d\l<M,N\r>_t,
\]
we arrive at the identity
\begin{align*}
&\int^{t_k}_0 X_r\, dY_r - \frac{1}{2}X_{t_k}Y_{t_k}\\
&= \int^{t_k}_0 X_r\, dY_r - \frac{1}{2}\bra{\int^{t_k}_0 X_r\, dY_r + \int^{t_k}_0 Y_r\, dX_r + \l<X,Y\r>_{0,t_k}}\\
&= \frac{1}{2}\int^{t_k}_0 \bra{X_r\, dY_r - Y_r\, dX_r} - \frac{1}{2}\l<X,Y\r>_{0,t_k}
= A_{0,t_k}^{M;i,j}- \frac{1}{2}\l<X,Y\r>_{0,t_k} 
= \hat{A}^{b,f;i,j}_{0,\frac{k}{n}}.
\end{align*}
Thus the desired convergence (\ref{e-area-goal-1}) for the case of $(\theta,\lambda)=(b,f)$ is deduced by applying the DCT to (\ref{e-therefore-2}). 
The proof is complete. 
\end{proof}

\section{Maximal $p$-variation bounds}\label{s-maximal}

The objective of this section is to prove the following maximal $p$-variation bound.

\begin{Proposition}[\textbf{Maximal $p$-variation bounds}]\label{p-maximal-p-variation-bound} 
	For all $p>2$, there exists a constant $C_p>0$ such that 
	\[
		\sup_{D\in\Dd_{[0,1]}} \EE\bra{ \pnorm{ \rp{M}^{D}}{p}^p} \leq C_p\norm{M}_p^p < \infty.
	\]
\end{Proposition}

We split the proof into two propositions concerning $M^{D,\theta,\lambda}$; the first dealing with $\theta=\lambda$ and the second $\theta\neq \lambda$, where $\theta,\lambda\in\seq{b,f}$. 

\begin{Proposition}\label{p-first-prop}
	For all $p>2$ and $\theta\in\seq{b,f}$, there exists a  constant $C_p>0$ such that 
	\[
		\sup_{D\in\Dd_{[0,1]}} \EE\bra{ \pnorm{S_2\bra{M^{D,\theta,\theta}}}{p}^p} \leq C_p\norm{M}^p_p.
	\]
\end{Proposition}
\begin{proof}
	For all $\theta\in \seq{b,f}$, $M^{D,\theta}$ is a reparameterization of the standard piecewise linear approximation $\tilde{M}^D$ of $M$ based on the partition $D$. Since $p$-variation is invariant under reparameterization, it follows from the Burkholder-Davis-Gundy rough path result of  \cite[Theorem 14.15]{FV}, (using the moderate function $F(x)=x^p$), that we have 
	\begin{align*}
		 \EE\bra{ \pnorm{S_2\bra{M^{D,\theta}}}{p}^p} &= \EE\bra{ \pnorm{S_2\bra{\tilde{M}^D}}{p}^p}
		 \leq C_p \EE\bra{ \qv{M}^{p/2}_{0,1}} \leq C_p\norm{M}_p^p.
	\end{align*}
	The proof is complete. 
\end{proof}

\begin{Remark}
	The use of \cite[Theorem 14.15]{FV} in the proof of Proposition \ref{p-first-prop} critically requires that $p>2$.
\end{Remark}

We now consider the case of $\theta\neq \lambda$. As in the proof of Proposition \ref{p-uniform-convergence}, by symmetry it suffices to only consider the case of $\theta=b\neq \lambda=f$. 

\begin{Proposition}\label{p-maximal-bf}
	For all $p>2$, there exists a constant $C_p>0$ such that 
	\[
		\sup_{D\in\Dd_{[0,1]}} \EE\bra{ \pnorm{S_2\bra{M^{D,b,f}}}{p}^p} \leq C_p\norm{M}_p^p.
	\]
\end{Proposition}
\begin{proof}
Since Proposition \ref{p-first-prop} has already established that 
\[
\sup_{D\in\Dd_{[0,1]}} \EE\bra{\norm{\pi_1\bra{S_2\bra{M^{D,b,f}}}}^p_{p\textup{-var;}[0,1]}} \leq C_p \norm{M}_p^p,
\]
it is sufficient to show that for all $D\in\Dd_{[0,1]}$, there exists a constant $C_p$ independent of $D$ such that 
\[
\EE\bra{\sup_{P\in \Dd_{[0,1]}} \sum_{s_j \in P} \abs{A^{D,b,f;i,j}_{s_j,s_{j+1}}}^{p/2}} 
= \EE\bra{ \abs{A^{D,b,f;i,j}}_{p/2\textup{-var;}[0,1]}^{p/2}}
\leq C_p \norm{M}_p^p.
\] 
We begin by fixing a partition $D=\{t_k\}_{k=0}^n\in\Dd_{[0,1]}$. As done previously, let $X:=M^{i}$ and $Y:=M^{j}$. 
By summing up the absolute area of the triangles enclosing the piecewise axis path $(M^{D,b;i},M^{D,f;j})$ over each interval $[\frac{2k}{2n},\frac{2k+1}{2n}]$, we find that 
\begin{align*}
\abs{A^{D,b,f;i,j}}_{p/2\textup{-var;}[0,1]} 
&\leq \abs{A^{D,b,f;i,j}}_{1\textup{-var;}[0,1]}\\
&= \sum_{k=0}^{n-1} \abs{A^{D,b,f;i,j}}_{1\textup{-var;}\l[\frac{2k}{2n},\frac{2k+2}{2n}\r]}
= \frac{1}{2}\sum_{t_k\in D} \abs{X_{t_k,t_{k+1}}} \cdot \abs{Y_{t_{k+1},t_{k+2}}}.
\end{align*}
Therefore, using the inequalities $ab\leq \frac{1}{2}(a^2+b^2)$ and $(a+b)^{r} \leq 2^{r-1}(a^r + b^r)$ for $a,b\geq 0$, $r\geq 1$, we find
\begin{align*}
\EE\bra{ \abs{A^{D,b,f;i,j}}_{p/2\textup{-var;}[0,1]}^{p/2}}
&\leq 
 \EE\bra{ \abs{A^{D,b,f;i,j}}_{1\textup{-var;}[0,1]}^{p/2}}\\
&= 2^{-p/2}\EE\bra{\abs{\sum_{t_k \in D} \abs{X_{t_k,t_{k+1}}} \cdot \abs{Y_{t_{k+1},t_{k+2}}}}^{p/2}}\\
&\leq \frac{1}{2}\EE\bra{\abs{\sum_{t_k\in D} \abs{X_{t_k,t_{k+1}}}^2}^{p/2}} + 
\frac{1}{2}\EE\bra{\abs{\sum_{t_k\in D} \abs{Y_{t_k,t_{k+1}}}^2}^{p/2}}. 
\end{align*}
From the discrete Burkholder-Davis-Gundy inequality (\cite[Theorem 1.1]{burkholder1972integral}), followed by the continuous BDG inequality, we deduce that 
\begin{align*}
\EE\bra{\abs{\sum_{t_k\in D} \abs{X_{t_k,t_{k+1}}}^2}^{p/2}}
\leq C_p\EE\bra{\sup_{t_k \in D} \abs{X_{t_k}-X_0}^p}
&\leq C_p\EE\bra{\sup_{t\in [0,1]} \abs{X_{t}-X_0}^p}\\
&\leq C_p\EE\bra{\l<X\r>_{0,1}^{p/2}} \leq C_p\norm{M}_p^p.
\end{align*}
An identical inequality holds for the sum of squares term involving $Y$. As this bound is independent of the original partition choice of $D$, the proof is complete. 
\end{proof}

\begin{Remark}
Instead of a martingale, suppose our underlying signal $X:[0,1]\to\RR$ is a fractional Brownian motion with Hurst parameter $H<\frac{1}{2}$. Then as $p>2$, 
\begin{align*}
\abs{ \sum_{t_k\in D} \abs{X_{t_k,t_{k+1}}}^2}_{L^{p/2}} \geq  
\abs{ \sum_{t_k\in D} \abs{X_{t_k,t_{k+1}}}^2}_{L^{1}}  &= \EE\bra{\sum_{t_k\in D} \abs{X_{t_k,t_{k+1}}}^2}
 = \sum_{t_k \in D} \abs{t_{k+1}-t_k}^{2H} \to \infty
\end{align*}
as $\abs{D}\to 0$. Thus the Burkholder-Davis-Gundy-based argument in the proof of Proposition \ref{p-maximal-bf}, (used to bound the individual area triangles enclosed by the Hoff path), cannot be generalized to the case where the underlying signal is a fractional Brownian motion (at least when $H<\frac{1}{2}$). This is the first obstacle we encounter when attempting to extend the convergence results of Theorems \ref{t-hoff-process-convergence} and \ref{t-ito-sde} to this non-semimartingale situation. Perhaps other techniques exploiting the properties of Gaussian rough paths are needed (see \cite[Chapter XV]{FV}), but the authors are uncertain on whether such an extension is even possible; the quadratic variation of fractional Brownian motion with Hurst index $H<\frac{1}{2}$ is infinite (\cite[\S2]{rogers1997arbitrage}). 
\end{Remark}

\section*{Acknowledgements}

The authors would like to thank Prof.~Mike Giles, Ni Hao, Sean Ledger and Weijun Xu for their helpful comments along with Prof.~Zhongmin Qian and Danyu Yang for reading an earlier draft. Thanks must also go to Lajos Gyurk\'{o} for helping the first author write a Python script to produce Figure \ref{f-hoff-greg}.
The research is supported by the European Research Council under the European Union's
Seventh Framework Programme (FP7-IDEAS-ERC, ERC grant agreement nr. 291244). 
The authors are grateful for the support of the Oxford-Man Institute.

\bibliographystyle{plain}
\bibliography{hoff_library}

\end{document}